\documentclass[letterpaper,12pt]{article}
\usepackage{amsmath}
\usepackage{amsfonts}
\usepackage{amssymb}
\usepackage{amsthm}
\usepackage{hyperref}
\usepackage{enumerate}
\usepackage[margin=1in]{geometry}
\usepackage[all,cmtip]{xy}

\newtheorem{theorem}{Theorem}[section]
\newtheorem{rem}[theorem]{Remark}
\newenvironment{remark}{\begin{rem}\rm}{\end{rem}}

\newtheorem{proposition}[theorem]{Proposition}

\newtheorem{lemma}[theorem]{Lemma}

\newtheorem{eg}[theorem]{Example}

\newtheorem{definition}[theorem]{Definition}

\newcommand{\C}{\mathbb{C}}

\newcommand{\R}{\mathbb{R}}

\renewcommand{\S}{\mathcal{S}}
\newcommand{\K}{\mathcal{K}}
\renewcommand{\phi}{\varphi}

\newcommand{\A}{\mathcal{A}}

\newcommand{\tpi}{\tilde{\pi}}

\DeclareMathOperator{\Span}{span}
\DeclareMathOperator{\End}{End}
\DeclareMathOperator{\Img}{im}
\DeclareMathOperator{\Id}{Id}
\DeclareMathOperator{\rank}{rank}

\title{On the geometry of almost $\S$-manifolds}
\author{Sean Fitzpatrick\footnote{Research supported by an NSERC postdoctoral fellowship}\\
749 Evans Hall \#3840\\
Department of Mathematics\\
University of California, Berkeley\\
Berkeley, California, USA 94720\\sean@math.berkeley.edu}

\begin{document}

\maketitle

\begin{abstract}
An $f$-structure on a manifold $M$ is an endomorphism field $\varphi$ satisfying $\varphi^3+\varphi=0$.  We call an $f$-structure {\em regular} if the distribution $T=\ker\phi$ is involutive and regular, in the sense of Palais.  We show that when a regular $f$-structure on a compact manifold $M$ is an almost $\S$-structure, it determines a torus fibration of $M$ over a symplectic manifold.  When $\rank T = 1$, this result reduces to the Boothby-Wang theorem.  Unlike similar results for manifolds with $\S$-structure or $\K$-structure, we do not assume that the $f$-structure is normal.  We also show that given an almost $\S$-structure, we obtain an associated Jacobi structure, as well as a notion of symplectization.
\end{abstract}

\section{Introduction}
Let $(M,\eta)$ be a cooriented contact manifold.  The Boothby-Wang theorem \cite{BW} tells us that if the Reeb field $\xi$ corresponding to the contact form $\eta$ is regular (in the sense of Palais \cite{Palais}), then $M$ is a prequantum circle bundle $\pi:M\to N$ over a symplectic manifold $(N,\omega)$, where $\pi^*\omega = -d\eta$, and $\eta$ may be identified with the connection 1-form.  Conversely, let $M$ be a prequantum circle bundle over a symplectic manifold $(N,\omega)$ and let $\eta$ be a connection 1-form.  Given a choice of compatible almost complex structure $J$ for $\omega$, let $G(X,Y)=\omega(JX,Y)$ be the associated Riemannian metric on $N$, and let $\tpi$ denote the horizontal lift of vector fields defined by $\eta$.  We can then define an endomorphism field $\phi\in\Gamma(M,\End(TM))$ by
\[
 \phi X = \tpi J\pi_* X,
\]
and a Riemannian metric $g$ by $g=\pi^*G+\eta\otimes\eta$.  If we let $\xi$ be the vertical vector field satisfying $\eta(\xi)=1$, then $(\phi,\xi,\eta,g)$ defines a contact metric structure on $M$ \cite{Blair3}.  In particular, we note that $\phi$ is an $f$-structure on $M$.  By construction, we have $\phi^2= -\Id_{TM} + \eta\otimes\xi$, from which it follows that $\phi^3+\phi=0$.

In \cite{Blair2,BLY}, Blair et al consider compact Riemannian manifolds equipped with a regular normal $f$-structure $\phi$, and show that such manifolds are the total space of a principal torus bundle over a complex manifold $N$, and that in addition, $N$ is a K\"ahler manifold if the fundamental 2-form of the $f$-structure is closed (that is, if $M$ is a $\K$-manifold).  Saenz argued in \cite{Saenz} that if this $\K$-structure is an $\S$-structure, then the symplectic form of the K\"ahler manifold $N$ is integral.

While the results in \cite{BLY,Saenz} provide us with a generalization of the Boothby-Wang theorem, the proofs in \cite{BLY} (and by extension, the argument in \cite{Saenz}) rely in several places on the assumption that the $f$-structure $\phi$ is normal.   Since this assumption is not required in the original Boothby-Wang theorem, it is natural to ask what can be said if this assumption is dropped for $f$-structures of higher corank. In this note, we use a theorem of Tanno \cite{Tan1} to show that if $M$ is a compact almost $\S$-manifold, in the sense of \cite{DIP}, then $M$ is a principal torus bundle over a symplectic manifold whose symplectic form is integral. (More precisely, the symplectic form will be a real multiple of an integral symplectic form.)  Not surprisingly, this tells us that requiring $\phi$ to be normal is the same as demanding that the base of our torus bundle be K\"ahler.

This ``generalized Boothby-Wang theorem'' is one of a number of similarities between manifolds with almost $\S$-structure and contact manifolds.  In the final section of this paper we demonstrate two more.  First, there is a natural notion of symplectization: given an almost $\S$-manifold $M$, there is an open, conic, symplectic submanifold of $T^*M$ whose base is $M$.  Second, a choice of one-form (expressed in terms of the almost $\S$-structure) allows us to define a Jacobi bracket on the algebra of smooth functions on $M$, giving us in particular a notion of Hamiltonian vector field on manifolds with almost $S$-structure.

\section{Preliminaries}
\subsection{Regular involutive distributions}
Let $F\subset TM$ be an involutive distribution of rank $k$.  We briefly recall the notion of a regular distribution in the sense of Palais, and refer the reader to \cite{Palais} for the details.  Roughly speaking, the involutive distribution $F$ is {\em regular} if each point $p\in M$ has a coordinate neighbourhood $(U, x^1,\ldots.x^n)$ such that
\[
\left\{\left(\dfrac{\partial}{\partial x^1}\right)_p,\ldots, \left(\dfrac{\partial}{\partial x^k}\right)_p\right\} 
\]
 forms a basis for $F_p\subset T_pM$, and such that the integral submanifold of $F$ through $p$ intersects $U$ in only one $k$-dimensional slice.  When $F$ is regular, the leaf space $\mathcal{F} = M/F$ is a smooth Hausdorff manifold, and the quotient mapping $\pi_F:M\to\mathcal{F}$ is smooth and closed.  When $M$ is compact and connected, the leaves of $F$ are compact and isomorphic, and are the fibres of the smooth fibration $\pi_F:M\to\mathcal{F}$.

In particular, a vector field $X$ on $M$ is regular if each $p\in M$ has a neighbourhood $U$ through which the integral curve of $X$ through $p$ passes only once.  If $M$ is compact, the integral curves of a regular vector field are thus diffeomorphic to circles.  Applying this fact to the Reeb vector field of a contact manifold gives part of the proof of the Boothby-Wang theorem.
\subsection{$f$-structures}\label{fstruct}
An {\em $f$-structure} on $M$ is an endomorphism field $\varphi\in\Gamma(M,\End TM)$ such that
\begin{equation}\label{fdef}
 \varphi^3+\varphi = 0.
\end{equation}
Such structures were introduced by K. Yano in \cite{Yano}; many of the facts regarding $f$ structures are collected in the book \cite{KY2}.    By a result of Stong \cite{Stong}, every $f$-structure is of constant rank.  If $\rank \varphi = \dim M$, then $\varphi$ is an almost complex structure on $M$, while if $\rank \varphi = \dim M-1$, then $\varphi$ determines an almost contact structure on $M$.

It is easy to check that the operators $l=-\varphi^2$ and $m=\varphi^2+\Id_{TM}$ are complementary projection operators; letting $E=l(TM) = \Img \varphi$ and $T=m(TM)=\ker \varphi$, we obtain the splitting
\begin{equation}\label{Tsplit}
 TM = E\oplus T = \Img \varphi\oplus \ker \varphi
\end{equation}
of the tangent bundle.  Since $(\varphi|_E)^2 = -\Id_E$, $\varphi$ is necessarily of even rank. When the corank of $\phi$ is equal to one, the distribution $T$ is automatically trivial and involutive.  However, if $\rank T>1$, this need not be the case, and one often makes additional simplifying assumptions about $T$.  An $f$-structure such that $T$ is trivial is called an {\em $f$-structure with parallelizable kernel} (or $f\cdot$pk-structure for short) in \cite{DIP}.  We will assume that an $f\cdot$pk-structure includes a choice of a trivializing frame $\{\xi_i\}$ and corresponding coframe $\{\eta^i\}$ for $T^*$, with
\[
 \eta^i(\xi_j) = \delta^i_j,\quad \varphi(\xi_i) = \eta^j\circ\varphi = 0,\quad\text{and}\quad \varphi^2 = -\Id + \sum\eta^i\otimes \xi_i.
\]
(This is known as an {\em $f$-structure with complemented frames} in \cite{Blair2}; such a choice of frame and coframe always exists.)  Given an $f\cdot$pk-structure, it is always possible \cite{KY2} to find a Riemannian metric $g$ that is compatible with $(\varphi,\xi_i,\eta^j)$ in the sense that, for all $X,Y\in \Gamma(M,TM)$, we have
\begin{equation}\label{phig}
 g(X,Y) = g(\varphi X,\varphi Y)+\sum_{i=1}^k \eta^i(X)\eta^i(Y).
\end{equation}
Following \cite{DIP}, we will call the 4-tuple $(\varphi,\xi_i,\eta^j, g)$ a {\em metric $f\cdot$pk structure}.  
Given a metric $f\cdot$pk-structure $(\varphi, \xi_i,\eta^j,g)$, we can define the {\em fundamental 2-form} $\Phi_g\in\A^2(M)$ by
\begin{equation}\label{2form}
 \Phi_g(X,Y) = g(\varphi X, Y).
\end{equation}
\begin{remark}
  Our definition of $\Phi_g$ is chosen to agree with our preferred sign conventions in symplectic geometry; however, many authors place $\varphi$ in the second slot, so our convention here uses the opposite sign of that found for example in \cite{BLY} and \cite{DIP}.
\end{remark}

We will call an $f$-structure $\phi$ {\em regular} if the distribution $T=\ker\phi$ is regular in the sense of Palais \cite{Palais}.  An $f\cdot$pk-structure is regular if the vector fields $\xi_i$ are regular and independent. An $f\cdot$pk-structure is called {\em normal} \cite{Blair2} if the tensor $N$ defined by
\begin{equation}\label{normalf}
 N= [\varphi,\varphi] + \sum_{i=1}^k d\eta^i\otimes \xi_i,
\end{equation}
vanishes identically. Here $[\varphi,\varphi]$ denotes the Nijenhuis torsion of $\varphi$, which is given by
\[
 [\varphi,\varphi](X,Y) = \varphi^2[X,Y]+[\varphi X,\varphi Y]-\varphi[\varphi X,Y]-\varphi[X,\varphi Y].
\]
When $\phi$ is normal, the $+i$-eigenbundle of $\phi$ (extended by $\C$ linearity to $T_\C M$) defines a CR structure $E_{1,0}\subset T_\C M$.
Regular normal $f$-structures are studied in \cite{BLY}, where it is proved that a compact manifold with regular normal $f$-structure is a principal torus bundle over a complex manifold $N$.  If the fundamental 2-form $\Phi_g$ of a normal $f$-structure is closed, then the $f$-structure is called a $\K$-structure, and $M$ a $\K$-manifold.  For a compact regular $\K$-manifold $M$, the base $N$ of the torus fibration is a K\"ahler manifold.  A special case of a $\K$-manifold is an $\S$-manifold.  On an $\S$ manifold there exist constants $\alpha^1,\ldots, \alpha^k$ such that $d\eta^i = -\alpha^i\Phi_g$ for $i=1,\ldots, k$.  Two commonly considered cases are the case $\alpha^i=0$ for all $i$, and the case $\alpha^i=1$ for all $i$.  In the language of CR geometry, the former case is analogous to a ``Levi-flat'' CR manifold, while the latter defines an analogue of a strongly pseudoconvex CR manifold (typically, strongly pseudoconvex CR manifolds are assumed to be of ``hypersurface type,'' meaning that the complementary distribution $T$ has rank one; see \cite{DT}).

A refinement of the notion of $\S$-structure was introduced in \cite{DIP}: a metric $f\cdot$pk-structure $(\phi,\xi_i,\eta^j,g)$ which is not necessarily normal is called an {\em almost} $\S$-structure if $d\eta^i=-\Phi_g$ for each $i=1,\ldots, k$.  An $f$-structure $\phi$ is called CR-integrable in \cite{DIP} if the $+i$-eigenbundle $E_{1,0}\subset T_\C M$ of $\phi$ is involutive (and hence, defines a CR structure).  It is shown in \cite{DIP} that an $f\cdot$pk-structure is CR-integrable if and only if the tensor $N$ given by \eqref{normalf} satisfies $N(X,Y)=0$ for all $X,Y\in\Gamma(M,E)$, where $E=\Img\phi$, whereas for a normal $f\cdot$pk-structure, $N$ must vanish for all $X,Y\in\Gamma(M,TM)$.  In \cite{LP} it is proved that a CR-integrable almost $\S$-manifold admits a canonical connection analogous to the Tanaka-Webster connection of a strongly pseudoconvex CR manifold.  For the relationship between this connection and the $\overline{\partial}_b$ operator of the corresponding tangential Cauchy-Riemann complex, as well as an application of this relationship to defining an analogue of geometric quantization for almost $\mathcal{S}$-manifolds, see \cite{F5}.

In this paper, we will define an almost $\K$-structure to be a metric $f\cdot$pk-structure for which $d\Phi_g=0$, and we will define an almost $\S$-structure more generally to be an almost $\K$-structure such that $d\eta^i = -\alpha^i\Phi_g$ for constants $\alpha^i\in\R$, for $i=1,\ldots, k$.

\section{Properties of almost $\K$ and almost $\S$-structures}
Let $(\phi,\xi_i,\eta^i)$ be an $f\cdot$pk-structure on a compact, connected manifold $M$.  Let $g$ be a Riemannian metric satisfying the compatibility condition \eqref{phig}, and let $\Phi_g$ denote the corresponding fundamental 2-form.  Let $E=\Img\phi$, and $T=\ker\phi$ denote the distribution spanned by the $\xi_i$.  It's easy to check that the distributions $E$ and $T$ are orthogonal with respect to $g$, and that the restriction of $\Phi_g$ to $E\otimes E$ is nondegenerate, from which we have the following: 
\begin{lemma}\label{lemmy}
  $X\in\Gamma(M,T)$ if and only if $\iota(X)\Phi_g = 0$.
\end{lemma}

\begin{proposition}\label{kstr}
 Let $(\phi,\xi_i,\eta^i,g)$ be a metric $f\cdot$pk-structure.  Then $T=\ker\phi$ is involutive whenever $d\Phi_g=0$.
\end{proposition}
\begin{proof}
 Let $X,Y\in\Gamma(M,T)$, and let $Z\in\Gamma(M,TM)$.  Then, using Lemma \ref{lemmy} above, we have
\begin{align*}
 d\Phi_g(X,Y,Z) & = X\cdot\Phi_g(Y,Z)+Y\cdot\Phi_g(Z,X)+Z\cdot\Phi_g(X,Y)\\
&\quad\quad\quad -\Phi_g([X,Y],Z)-\Phi_g([Y,Z],X) - \Phi_g([Z,X],Y)\\
&= -\Phi_g([X,Y],Z).
\end{align*}
Therefore, if $d\Phi_g = 0$, then $\iota([X,Y])\Phi_g=0$, and thus  $[X,Y]\in \Gamma(M,T)$, which proves the proposition.
\end{proof}
Let us now suppose that $(\phi,\xi_i,\eta^i,g)$ is an almost $\S$-structure, so that the 1-forms $\eta^i$ satisfy $d\eta^i = -\alpha^i\Phi_g$ for constants $\alpha^i$, some of which may be zero.  The following results were proved in \cite{DIP} in the case  that $\alpha^i=1$ for all $i$; we easily see that the results remain true in our more general setting:
\begin{proposition}\label{aa}
 If $(\phi,\xi_i,\eta^j,g)$ is an almost $\S$-structure, then $\mathcal{L}(\xi_i)\xi_j=[\xi_i,\xi_j]=0$ for all $i,j=1,\ldots,k$.
\end{proposition}
\begin{proof}
 Since the fundamental 2-form $\Phi_g$ of an almost $\S$-structure is closed, the distribution $T$ is involutive.  Thus we may write $[\xi_i,\xi_j] = \sum c_{ij}^a\xi_a$. But for any $a,i,j\in \{1,\ldots,k\}$, we have
\[
 c^a_{ij} = \eta^a([\xi_i,\xi_j]) = \xi_i\cdot\eta^a(\xi_j)-\xi_j\cdot\eta^a(\xi_i)-d\eta^a(\xi_i,\xi_j) = \alpha^a\Phi_g(\xi_i,\xi_j)=0. \qedhere
\]
\end{proof}
\begin{proposition}\label{ab}
If $(\phi,\xi_i,\eta^j,g)$ is an almost $\S$-structure, then $\mathcal{L}(\xi_i)\eta^j=0$ for all $i,j=i,\ldots,k$.
\end{proposition}
\begin{proof}
 We have $\mathcal{L}(\xi)\eta^j = d(\eta^j(\xi_i)) + \iota(\xi_i)d\eta^j = -\alpha^j(\iota(\xi_i)\Phi_g)=0$.
\end{proof}
We remark that several other results from \cite{DIP} hold in this more general setting, but they are not needed here.  To conclude this section, we state a theorem due to Tanno \cite{Tan1}:
\begin{theorem}\label{Tanno}
 For a regular and proper vector field $X$ on a manifold $M$, the following are equivalent:
\begin{enumerate}[(i)]
 \item The period function $\lambda_X$ of $X$ is constant.
 \item There exists a 1-form $\eta$ such that $\eta(X)=1$ and $\mathcal{L}(X)\eta = 0$.
 \item There exists a Riemannian metric $g$ such that $g(X,X)=1$ and $\mathcal{L}(X)g=0$.
\end{enumerate}
\end{theorem}
In the above theorem, the period function $\lambda_X:M\to \R$ is defined by
\begin{equation}
 \lambda_X(p) = \inf\{t>0|\exp(tX)\cdot p = p\}.
\end{equation}
If $M$ is noncompact, the value $\lambda_X(p)=\infty$ is possible.  Part (iii) of the above tells us that $X$ is a unit Killing field for the metric $g$.  Using this result, Tanno was able to give a simple proof (which is reproduced in \cite{Blair3}) of the Boothby-Wang theorem \cite{BW}.  

\section{The structure of regular almost $\S$-manifolds}
As noted above, from \cite{BLY}, a compact manifold with regular normal $f$-structure is a principal torus bundle over a complex manifold $N$, and $N$ is K\"ahler if $M$ is a $\K$-manifold.  If $M$ is an $\S$-manifold with $\Phi_g=-d\eta^i$ for each $i$, then by \cite{Saenz}, the symplectic form on $N$ is integral.  We now dispense with the requirement that the $f$-structure on $M$ be normal, and state a similar result for almost $\S$-manifolds.
\begin{theorem}\label{main}
 Let $M$ be a compact manifold of dimension $2n+k$ equipped with a regular almost $\S$-structure $(\phi,\tilde{\xi}_i,\tilde{\eta}^i,\tilde{g})$ of rank $2n$.  Then there exists an almost $\S$-structure $(\phi,\xi_i,\eta^i,g)$ on $M$ for which the vector fields $\xi_1,\ldots,\xi_k$ are the infinitesimal generators of a free and effective $\mathbb{T}^k$-action on $M$. Moreover, the quotient $N=M/\mathbb{T}^k$ is a smooth symplectic manifold of dimension $2n$, and if the $\alpha^i$ such that $d\tilde{\eta}^i = -\alpha^i\Phi_{\tilde{g}}$ are not all zero, then the symplectic form on $N$ is a real multiple of an integral symplectic form.
\end{theorem}
\begin{proof}
 By assumption, the vector fields $\tilde{\xi}_1,\ldots, \tilde{\xi}_k$ are regular, independent and proper, and by Proposition \ref{kstr}, the distribution $T=\Span\{\tilde{\xi}_1,\ldots,\tilde{\xi}_k\}$ is involutive.  Thus, by the results of Palais, $N=M/T$ is a smooth manifold, and $\pi:M\to N$ is a smooth fibration whose fibres are the leaves of the distribution $T$.  Since $M$ is compact, the fibres are compact and isomorphic \cite{Palais}.  For each $i=1,\ldots,k$, we have $\tilde{\eta}^i(\tilde{\xi}_i)=1$ and $\mathcal{L}(\tilde{\xi}_i)\tilde{\eta}^i = 0$.  Thus, by Theorem \ref{Tanno}, the period functions $\lambda_i = \lambda_{\tilde{\xi}_i}$ are constant.  We rescale by setting $\xi_i = \lambda_i\tilde{\xi}_i$ and $\eta^i=\frac{1}{\lambda_i}\tilde{\eta}^i$.  We still have $\eta^i(\xi_j)=\delta^i_j$, and note that the associated metric $g$ for which $(\phi,\xi_i,\eta^i,g)$ is an almost $\S$-structure differs from $\tilde{g}$ only along $T$, so that $\Phi_g = \Phi_{\tilde{g}}$.  Each $\xi_i$ now has period 1, and since the vector fields $\xi_i$ all commute, they are the generators of a free and effective $\mathbb{T}^k$-action on $M$.  The argument for local triviality is the same as in \cite{BLY}, so we do not repeat it here.  Thus, we have that $M$ is a principal $\mathbb{T}^k$-bundle over $N=M/T$.  The infinitesimal action of $\R^k$ is given by
\[
 X=(t^1,\ldots, t^k)\mapsto X_M = \sum t^i\xi_i,
\]
from which we see that $\boldsymbol{\eta} = (\eta^1,\ldots, \eta^k)$ is a connection 1-form on $M$: we have $\iota(X_M)\eta = X$ and $\mathcal{L}(X_M)\eta = 0$ for all $X\in \R^k$.

Now, we note that the fundamental 2-form $\Phi_g$ is horizontal and invariant, since $\iota(X)\Phi_g = \mathcal{L}(X)\Phi_g = 0$ for all $X\in\Gamma(M,T)$, and thus there exists a 2-form $\Omega$ on $N$ such that $\pi^*\Omega = \Phi_g$.  Since $\pi^*d\Omega = d\Phi_g = 0$, $\Omega$ is closed, and since $\pi^*\Omega^n = \Phi_g^n \neq 0$, $\Omega$ is non-degenerate, and hence symplectic.

Finally, let us suppose that one of the $\alpha^i$ are non-zero; without loss of generality, let's say $\alpha^1\neq 0$.  By the same argument as above, the vector fields $\xi_2,\ldots, \xi_k$ generate a free $\mathbb{T}^{k-1}$-action on $M$, giving us a fibration $p:M\to P$.  Now, since $\mathcal{L}(\xi_i)\xi_1 = \mathcal{L}(\xi_i)\eta^1 = 0$ for $i=2,\ldots, k$, the vector field $\xi_1$ and 1-form $\eta^1$ are invariant under the $\mathbb{T}^{k-1}$-action.  We can thus define a 1-form $\eta$ on $P$ by $\eta(X) = \eta^1(\tilde{p}X)$, where $\tilde{p}X$ denotes the horizontal lift of $X$ with respect to the connection 1-form defined by $\eta^2,\ldots, \eta^k$, and a vector field $\xi$ on $P$ by $\xi = p_*\xi_1$.  Note that $d\eta(X,Y) = d\eta^1(\tilde{p}X,\tilde{p}Y)$.  We then have $\eta(\xi) = 1$, and $\mathcal{L}(\xi)\eta = \iota(\xi^1)d\eta^1 = 0$, so that Theorem \ref{Tanno} applies to the pair $(\eta,\xi)$.  It follows that $\xi$ generates a free action of $S^1=\mathbb{R}/\mathbb{Z}$ on $P$, giving us the $\mathbb{T}^1$-bundle structure $q:P\to N$.  Since $\pi = q\circ p$, it follows that
\[
 d\eta(X,Y) = d\eta^1(\tilde{p}X,\tilde{p}Y) = -\frac{\alpha^1}{\lambda_1}(\pi^*\Omega)(\tilde{p}X,\tilde{p}Y) = -\frac{\alpha^1}{\lambda_1}q^*\Omega(X,Y).
\]
Thus, $P$ is a Boothby-Wang fibration over $(N,\frac{\alpha^1}{\lambda_1}\Omega)$, from which it follows that the symplectic form $\frac{\alpha^1}{\lambda}\Omega$ must be integral (see \cite{Kob}), and hence $\Omega$ is a real multiple of an integral symplectic form.
\end{proof}
\begin{remark}
 Note that since the last part of the argument is valid for any pair of nonzero constants $\alpha^i,\alpha^j$, from which it follows that for each $i,j$ for which $\alpha^i$ and $\alpha^j$ are nonzero, we must have $\dfrac{\alpha^i}{\lambda_i}\cdot\dfrac{\lambda_j}{\alpha^j}\in\mathbb{Q}$.
\end{remark}

Conversely, we have the following theorem:
\begin{theorem}
Suppose that $M$ is a principal $\mathbb{T}^k$-bundle over a symplectic manifold $(N,\omega)$, equipped with connection 1-form $\boldsymbol{\eta} = (\eta^1,\ldots, \eta^k)$ such that there exist constants $\alpha^1, \ldots, \alpha^k$ for which $d\eta^i = -\alpha^i\pi^*\omega$.  Then $M$ admits an almost $\S$-structure.
\end{theorem}
\begin{proof}
The proof is essentially the same as the proof given in \cite{Blair2} when $N$ is K\"ahler, if we omit the proof of normality.  Given a choice of compatible almost complex structure $J$ and associated metric $G$, we can define an $f$-structure $\varphi$ by $\varphi X = \tilde{\pi}J\pi_* X$, where $\tilde{\pi}$ denotes the horizontal lift with respect to $\boldsymbol{\eta}$.  If we let $\xi_1,\ldots, \xi_k$ denote vertical vectors such that $\eta^i(\xi_j) = \delta^i_j$, and define the metric $g$ by
\[
 g(X,Y) = \pi^*G(X,Y)+\sum\eta^i(X)\eta^i(Y),
\]
then it's straightforward to check that the data $(\varphi,\xi_i,\eta^j,g)$ defines an almost $\S$-structure on $M$.  (Note that $\Phi_g = \pi^*\omega$, so that $d\eta^i = -\alpha^i\Phi_g$.)
\end{proof}
\begin{remark}
 We can also use the results of Tanno \cite{Tan1} to show that the vector fields $\xi_1,\ldots, \xi_k$ of an almost $\S$-structure are Killing.  Let $\tpi$ denote the horizontal lift defined by $\boldsymbol{\eta}$.  Then we can define a Riemannian metric $G$ on $N$ by $G(X,Y) = g(\tpi X,\tpi Y)$ for any $X,Y\in\Gamma(N,TN)$, where $g$ is the metric of the almost $\S$-structure on $M$.  It follows that $g = \pi^*G + \sum \eta^i\otimes\eta^i$, whence $g(\xi_i,\xi_i)=1$ and $\mathcal{L}(\xi_i)g = 0$ for $i=1\ldots, k$.  Moreover, the endomorphism field $J\in\Gamma(N,\End(TN))$ defined by $JX = \pi_*\phi\tpi X$ is easily seen to be an almost complex structure on $N$ that is compatible with $G$, and the symplectic form $\Omega$ then satisfies $\Omega(X,Y) = G(X,JY)$.
\end{remark}
\begin{remark}
 If $M$ is only an almost $\K$-manifold, it is not clear that we can expect any analogous result to hold, since the proof in \cite{BLY} for a $\K$-manifold does not work without normality, and Tanno's theorem cannot be applied if $\mathcal{L}(\xi_i)\eta^j \neq 0$ for all $i,j$, and this need not hold if $d\eta^j$ is not a multiple of $\Phi_g$.
\end{remark}
\begin{remark}
 If $M$ is noncompact, then as noted below the statement of Tanno's theorem, the period $\lambda_i$ of one of the $\xi_i$ could be infinite, in which case $\xi_i$ generates an $\R$-action on $M$ instead of an $S^1$-action.  
\end{remark}
\section{Symplectization and Jacobi structures}
We conclude this paper with a discussion of the relationship between almost $\S$-structures and related geometries intended to reinforce the view that almost $\S$-structures deserve to be viewed as higher corank analogues of contact structures.  (However, see also \cite{vE2} for the notion of $k$-contact structures, which, from the point of view of Heisenberg calculus, are also deserving of the title of higher corank contact structure.  From this perspective, almost $\S$-structures are perhaps more analogous to contact metric structures, or even strongly pseudoconvex CR structures, although they are not CR-integrable in general.)

Recall that a stable complex structure on a manifold $M$ is a complex structure defined on the fibres of $TM\oplus \R^k$ for some $k$.  Given an $f\cdot$pk-structure $(\phi,\xi_i,\eta^j)$ on $M$, we obtain a stable complex structure $J\in\Gamma(M,\End(TM\oplus\R^k))$ by setting $JX=\phi X$ for $X\in\Gamma(M,E)$, and defining $J\xi_i = \tau_i$ and $J\tau_i = -\xi_i$, where $\tau_1,\ldots,\tau_k$ is a basis for $\R^k$.  As explained in \cite{GGK}, a stable complex structure determines a Spin$^c$-structure on $M$.

Alternatively, (and with some abuse of notation), we can think of the above complex structure on each fibre $T_xM\times \R^k$ as coming from an almost complex structure on $M\times\R^k$ obtained from to the $f$-structure $\phi$.  With this point of view, we note that  it is possible to define a ``symplectization'' analogous to the symplectization of a cooriented contact manifold, provided that our $f\cdot$pk-structure is an almost $\S$-structure, with at least one of the $\alpha^j$ (such that $d\eta^j=-\alpha^j\Phi_g$) nonzero.  As above, we let $TM = E\oplus T$ denote the splitting of the tangent bundle determined by the $f$-structure, and let $E^0\cong T^*=\Span\{\eta^i\}\cong M\times \R^k$ denote the annihilator of $E$.  It is then possible to find an open connected symplectic submanifold $E^0_+$ of $T^*M$ whose tangent bundle is $T_xM\times \R^k$. 
For concreteness, let us use the identification $E^0\cong M\times \R^k$, and with respect to coordinates $(x,t_1,\ldots, t_k)$, let
\[
 \alpha = \sum_{i=1}^k t_i\eta^i,
\]
and define $\omega = -d\alpha$.  (We are abusing notation here slightly; technically we should write $\pi^*\eta^i$ in place of $\eta^i$, where $\pi:M\times\R^k\to M$ is the projection onto the first factor.) Using the fact that $d\eta^i = -\alpha^i\Phi_g$ for each $i$, we have
\[
 \omega = \sum \eta^j\wedge dt_j +\left(\sum t_j\alpha^j\right)\Phi_g.
\]
Define $\tau\in C^\infty(E^0)$ to be the function given in coordinates by $\tau = \sum \alpha^j t_j$.  Note that since $\eta^i\wedge\eta^i = dt_i\wedge dt_i=0$, we have
\[
 \left(\sum_{i=1}^k \eta^j\wedge dt_j\right)^k = k!\,\eta^1\wedge dt_1\wedge\cdots \wedge\eta^k\wedge dt_k.
\]
We also note that $\Phi_g^m=0$ for $m>n$. Thus, using the binomial theorem, we find that the top-degree form $\omega^{n+k}$ has only one nonzero term; namely,
\[
 \omega^{n+k} = \frac{(n+k)!}{n!}\eta^1\wedge dt_1\cdots\wedge \eta^k\wedge dt_k\wedge(\tau\Phi_g)^n.
\]
Thus, $\omega^{n+k}$ is a volume form on the open subset $E^0_+$ of $E^0$ defined by $\tau>0$, and hence $\omega$ is a symplectic form on $E^0_+$.

Next, we will show that for certain choices of section $\eta\in \Gamma(M,E^0)$ we obtain a Jacobi structure on $M$ defined in a manner analogous to the Jacobi structure associated to a choice of contact form on a contact manifold.    We recall that a Jacobi structure on $M$ is given by a Lie bracket $\{\cdot,\cdot\}$ on $C^\infty(M)$ such that for any $f,g\in C^\infty(M)$ the support of $\{f,g\}$ is contained in the intersection of the supports of $f$ and $g$.  Jacobi structures were introduced independently by Kirillov \cite{Kir2} and Lichnerowicz \cite{Lich}; a good introduction can be found in \cite{Marle}.  

Again, we assume $M$ is equipped with an almost $\S$-structure with the constants $\alpha^j$ such that $d\eta^j = -\alpha^j\Phi_g$ not all zero.  Our first goal is to define a notion of a Hamiltonian vector field $X_f$ associated to each function $f\in C^\infty(M)$.  To begin with, let $\xi = \sum b^j\xi_j$ be an arbitrary section of $T=\ker\phi$, and let $\eta = \sum c_j\eta^j$ be an arbitrary section of $E^0\cong T^*$.  We will narrow down the possibilities for $\xi$ and $\eta$ as we consider the properties we wish the vector fields $X_f$ to satisfy.  The idea is to generalize the approach used to define Hamiltonian vector fields on a contact manifold $(M,\eta)$.  Recall that on manifold equipped with a contact form $\eta$, where we define $\Phi = -d\eta$, the Reeb vector field $\xi$ is defined by $\iota(\xi)\eta = 1$ and $\iota(\xi)\Phi = 0$.  A contact Hamiltonian vector field $X_f$ satisfies the equations $\iota(X_f)\eta = f$ and $\iota(X_f)\Phi = df - (\xi\cdot f)\eta$.  Lichnerowicz showed in \cite{Lich1} that these are the necessary and sufficient conditions for each $X_f$ to be an infinitesimal symmetry of the contact structure: it follows that for each $f\in C^\infty(M)$, $\mathcal{L}(X_f)\eta = (\xi\cdot f)\eta$.

We wish to impose similar conditions on $\xi$, $\eta$ and (the yet to be defined) $X_f$ in the case of almost $\S$-manifolds.  We already know that $\iota(\xi)\Phi_g = 0$, by Lemma \ref{lemmy}, so we begin by adding the requirement that $\eta(\xi)=\sum b^jc_j = 1$.  Next, we give our definition of a Hamiltonian vector field:
\begin{definition}
 Let $\eta$ and $\xi$ be as above.  For any $f\in C^\infty(M)$, we define the {\em Hamiltonian vector field} associated to $f$ by the equations
\begin{align}
 \iota(X_f)\eta^j &= \alpha^j f, \text{\rm for } j=1,\ldots, k,\label{ham1}\\ 
 \iota(X_f)\Phi_g &= df - (\xi\cdot f)\eta.\label{ham2}
\end{align}
\end{definition}
\begin{remark}
 Note that the above equations uniquely define $X_f$, by the nondegeneracy of the restriction of $\Phi$ to $E=\Img \phi$.  The constants $\alpha^j$ are the same ones such that $d\eta^j = -\alpha^j\Phi_g$.  One can check that if we began with $a^j$ in place of the $\alpha^j$, we would be forced to take $a^j=\alpha^j$ for consistency reasons.  (In particular this will be necessary if the bracket we define below is to be a Lie bracket.)  Moreover, this gives us the identity
\[
 \mathcal{L}(X_f)\eta^j = \alpha^j(\xi\cdot f)\eta
\]
for each $j=1,\ldots, k$; we would otherwise have an unwanted term of the form $(a^j-\alpha^j)df$.  Note that on the right-hand side of the above equation we have $\eta$ and not $\eta^j$; this is unavoidable with our definition of $X_f$.
\end{remark}
We can fix the coefficients of $\xi$ by requiring that $\xi$ be the Hamiltonian vector field associated to the constant function 1, as is standard for Jacobi structures (see \cite{Marle}).  It is easy to see that \eqref{ham1} then immediately forces us to take $\xi = \sum \alpha^j\xi_j$; that is, the coefficients $b^j$ are equal the constants $\alpha^j$.  Thus, $\xi$ is essentially determined by the almost $\S$ structure, although $\eta$ is constrained only by the condition $\eta(\xi)=1$, so the Jacobi structure we define below cannot be considered entirely canonical (as one might expect).  From the requirement that $\eta(\xi)=1$ it follows that for each $f\in C^\infty(M)$, we have
\[
 \mathcal{L}(X_f)\eta = \sum c_j \mathcal{L}(X_f)\eta^j = \sum c_j\alpha^j (\xi\cdot f)\eta = (\xi\cdot f)\eta,
\]
again in analogy with the contact case. Note that the normalization $\eta(\xi)=1$ also implies that $d\eta = -\Phi_g$. We are now ready to define our bracket on $C^\infty(M)$.
\begin{definition}
Let $M$ be a manifold with almost $\S$-structure, with constants $\alpha^j$ not all zero.  Let $\xi = \sum \alpha^j\xi_j$, and let $\eta$ be a section of $E^0$ such that $\eta(\xi)=1$.  We then define a bracket on $C^\infty(M)$ by
\begin{equation}\label{jacobi}
 \{f,g\} = \iota([X_f,X_g])\eta.
\end{equation}
\end{definition}
The bracket is clearly antisymmetric, and one checks (using the identity $\iota([X,Y]) = [\mathcal{L}(X),\iota(Y)]$) that
\[
 \{f,g\} = X_f\cdot g - X_g\cdot f + \Phi_g(X_f,X_g) = X_f\cdot g - (\xi\cdot f)g.
\]
Note that since the definition of the Hamiltonian vector fields depended on the choice of $\eta$, the bracket depends on $\eta$, even though $\eta$ no longer appears explicitly in either of the above expressions for the bracket.  From the latter equality we see that the support of $\{f,g\}$ is contained in the support of $g$, and by antisymmetry it must be contained in the support of $f$ as well.  Thus, the bracket given by \eqref{jacobi} is a Jacobi bracket provided we can verify the Jacobi identity.  Since the Jacobi identity is valid for the Lie bracket on vector fields, it suffices to prove the following:
\begin{proposition}\label{hamprop}
 Let $\{f,g\}$ be the bracket on $C^\infty(M)$ given by \eqref{jacobi}.
 Then the vector field $X_{\{f,g\}}$ corresponding to the function $\{f,g\}$ is given by $X_{\{f,g\}} = [X_f,X_g]$.
\end{proposition}
\begin{lemma}\label{LL1}
 For each $i=1,\ldots, k$, we have $[\xi_i,X_f] = X_{\xi_i\cdot f}$.
\end{lemma}
\begin{proof}
 From Propositions \ref{aa} and  \ref{ab}, we know that $[\xi_i,\xi_j]=0$ and $\mathcal{L}(\xi_i)\eta^j = 0$ for any $i,j\in\{1,\ldots,k\}$; from the latter it follows easily that $\mathcal{L}(\xi_i) \Phi_g= 0$ as well.  The result then follows from the uniqueness of Hamiltonian vector fields, since
\[
\iota([\xi_i,X_f])\eta^j = [\mathcal{L}(\xi_i),\iota(X_f)]\eta^j = \alpha^j \xi_i\cdot f, 
\]
and
\[
 \iota([\xi_i,X_f])\Phi_g = \mathcal{L}(\xi_i)(df - (\xi\cdot f)\eta) = d(\xi_i\cdot f) - (\xi\cdot(\xi_i\cdot f))\eta.\qedhere
\]
\end{proof}
\begin{lemma}\label{LL2}
 For each $i=1,\ldots, k$, we have $\xi_i\cdot\{f,g\} = \{\xi_i\cdot f, g\} + \{f, \xi_i\cdot g\}$.
\end{lemma}
\begin{proof}
 We have, using Lemma \ref{LL1} and the fact that $[\xi_i,\xi]=0$ in the second line,
\begin{align*}
\xi_i\cdot\{f,g\} & = \xi_i\cdot (X_f\cdot g) - \xi_i\cdot ((\xi\cdot f)g)\\
& =  X_f\cdot (\xi_i\cdot g) - (\xi\cdot f)(\xi_i\cdot g) + X_{\xi_i\cdot f}\cdot g - \xi(\xi_i\cdot f)g\\
& = \{f,\xi_i\cdot g\} +  \{\xi_i\cdot f,g\}.\qedhere
\end{align*}
\end{proof}
\begin{proof}[Proof of Proposition \ref{hamprop}]
 We need to show that $\iota([X_f,X_g])\eta^j = \alpha^j\{f,g\}$ for each $j=1,\ldots k$, and that $\iota([X_f,X_g])\Phi = d\{f,g\} - (\xi\cdot\{f,g\})\eta$.  First, since $\iota(X_g)\eta = \sum c_j\alpha^j g = g$, we have
\begin{align*}
 \iota([X_f,X_g])\eta^j & = \mathcal{L}(X_f)\eta^j(X_g) - \iota(X_g)\mathcal{L}(X_f)\eta^j\\
&=\alpha^j X_f\cdot g - \iota(X_g)(\alpha^j(\xi\cdot f)\eta = \alpha^j\{f,g\}.
\end{align*}
From Lemma \ref{LL2}, we have $\xi\cdot\{f,g\} = \{f,\xi\cdot g\}-\{g,\xi\cdot f\} = X_f\cdot(\xi\cdot g) - X_g\cdot(\xi\cdot f)$, and thus,
\begin{align*}
 \iota([X_f,X_g])\Phi_g & = \mathcal{L}(X_f)(dg - (\xi\cdot g)\eta) - \iota(X_g)(-d(\xi\cdot f)\wedge\eta + (\xi\cdot f)\Phi_g)\\
& = d(X_f\cdot g) - X_f\cdot(\xi\cdot g) - (\xi\cdot g)(\xi\cdot f)\eta+ X_g\cdot(\xi\cdot f)\eta\\
& \quad \quad  - gd(\xi\cdot f) - (\xi\cdot f)(dg-(\xi\cdot g)\eta)\\
& = d(X_f\cdot g - (\xi\cdot f)g) - (X_f\cdot(\xi\cdot g) - X_g\cdot(\xi\cdot f))\eta\\
& = d\{f,g\} - \xi\cdot\{f,g\}\eta.\qedhere
\end{align*}
\end{proof}
\section*{Acknowledgements}
The research for this article was made possible by a postdoctoral fellowship from Natural Sciences and Engineering Research Council of Canada, and by the University of California, Berkeley, the host institution for the fellowship.  The author would like to thank Alan Weinstein for several useful discussions and suggestions which helped to improve the paper.


\bibliographystyle{plain}
\bibliography{reference}
\end{document}